\documentclass[12pt]{amsart}

\newtheorem{theorem}{Theorem}[section]

\usepackage{tcolorbox}
\usepackage{xcolor}

\usepackage{pdfsync}
\usepackage{geometry}
\usepackage{lscape}
\usepackage{graphicx}
\usepackage{epstopdf}    
\usepackage[all]{xy}
\usepackage{amsmath,amsthm,amssymb,enumerate}
\usepackage{latexsym}
\usepackage{amscd}
\usepackage{amssymb}
\usepackage{enumitem}
\usepackage[mathscr]{euscript}

\theoremstyle{plain}  
\newtheorem{equ}[theorem]{Equation}
\newtheorem{diag}[theorem]{Diagram}
\newtheorem{thm}[theorem]{Theorem}
\newtheorem{prop}[theorem]{Proposition}
\newtheorem{claim}[theorem]{Claim}
\newtheorem{cor}[theorem]{Corollary}
\newtheorem{lemma}[theorem]{Lemma}
\newtheorem{conj}[theorem]{Conjecture}

\newtheorem{remark}[theorem]{Remark}

\newcommand{\ul}{\underline}

\newcommand{\ra}{\rightarrow}

\newcommand{\lra}{\longrightarrow}

\newcommand{\BPP}[1]{ \overline{\ul{BP}}_{\; #1}}
\newcommand{\SBPP}{\overline{BP}}
\newcommand{\BoP}[1]{ \ul{BoP}_{\;#1}}
\newcommand{\ZZ}[1]{ \ul{Z}_{\;#1}}
\newcommand{\ZZZ}[1]{ \ul{Z}'_{\;#1}}
\newcommand{\X}[1]{ \ul{X}_{\;#1}}

\newcommand{\bu}[1]{ \ul{bu}_{\;#1}}
\newcommand{\KO}[1]{ \ul{KO}_{\;#1}}
\newcommand{\bo}[1]{ \ul{bo}_{\;#1}}
\newcommand{\F}[1]{ \ul{F}_{\;#1}}

\newcommand{\Z}{\mathbb Z}

\newcommand{\BPn}[2]{\ul{BP\langle  #1  \rangle}_{\; #2}}
\newcommand{\SBPn}[1]{BP\langle  #1  \rangle}
\newcommand{\BoPn}[2]{ \ul{BoP\langle #1 \rangle}_{\;#2}}
\newcommand{\SBoPn}[1]{ BoP\langle #1 \rangle}
\newcommand{\SBPPn}[1]{\overline{BP}\langle #1 \rangle}

\newcommand{\Zp}{\Z_{(2)}}
\newcommand{\Zq}{\Z/(2)}

\DeclareMathOperator{\Tor}{Tor}

\newcommand{\sA}{\mathscr{A}}

 \newcommand{\bigSigma} {\mbox{\LARGE{$\Sigma$}}}
 
 \DeclareRobustCommand{\bigWedge} {\bigvee}


\begin{document}

\title{The Omega spectrum for Pengelley's $BoP$}

\author{W. Stephen Wilson}

\email{wwilson3@jhu.edu}

\address{Department of Mathematics, 
The Johns Hopkins University, 
Baltimore, MD 21218}


\begin{abstract}
We compute the homology of the spaces in the Omega spectrum for $BoP$. 
There is no torsion in $H_*(\BoP{i})$ for $i \ge 2$, and 
things are only slightly more complicated for $i < 2$.
We find the complete homotopy type of $\BoP{i}$ for $i \le 6$ and conjecture the 
homotopy type for $i  > 6$.
This completes the computation of all $H_*(\ul{MSU}_{\;*})$.
\end{abstract}


\maketitle

\section{Context}

There are several standard (co)bordism theories and related
spectra.  We have, for example, unoriented (co)bordism, $MO$,
oriented (co)bordism, $MSO$, Spin (co)bordism, $MSpin$, complex
(co)bordism, $MU$, and special unitary cobordism, $MSU$.
These theories are associated with Omega spectra classifying them,
i.e. $M^k(X) = [X,
\underline{M}_{\;k}] 
$
with $\Omega 
\underline{M}_{\;k+1}
=
\underline{M}_{\;k}
$.
We are only interested in the $p=2$ versions so all spectra
should be considered localized at 2.  We will suppress the notation
and let $M = M_{(2)}$.

Thom, \cite[1954]{Thom}, computed the coefficients $MO_*$ and
gave the stable homotopy type of the spectrum $MO$.  It is just
the product of a lot of mod 2 Eilenberg-MacLane spectra.
This gives the complete homotopy type of the 
$\underline{MO}_{\;k}$
as well.  From Serre's computation, \cite[1953]{SerreEM}, 
of the cohomology of the
Eilenberg-MacLane spaces, we also know $
H_*(\underline{MO}_{\;k})$ (mod 2).
Wall, \cite[1960]{Wall}, showed $MSO$ (at 2) is the
stable product of integral and mod 2 Eilenberg-MacLane
spectra.  Again, this gives the complete homotopy type of
the $\underline{MSO}_{\;k}$ and, by Serre, the homology of these
spaces.
Anderson, Brown, and Peterson, \cite[1967]{ABP}, computed $MSpin_*$
and gave the stable homotopy type, and consequently, the unstable
homotopy type of $\underline{MSpin}_{\;k}$.  Stably, this is 
copies of the mod 2 Eilenberg-MacLane spectra and connected covers
of the spectrum $bo$.  Stong, \cite[1963]{Stong}, computed the cohomology
of (most of) the unstable connected covers for $bo$, giving  (most of)
$H_*(\underline{MSpin}_{\;k})$.
At $p=2$, Milnor, \cite[1960]{Mil:MU}, and Novikov, \cite[1962]{Nov:MU2},
both computed $MU_*$.  With the construction of the Brown-Peterson
spectrum, \cite[1966]{BP}, the stable homotopy of $MU$ was described
as a product of $BP$ spectra.
In \cite[1973]{WSWThesis}, the homology of $\underline{BP}_{\;k}$ was
computed and in \cite[1975]{WSWThesis2} a complete description of
the unstable homotopy type was given.

$MSU$ has taken more time.  Conner and Floyd finished the computation
of $MSU_*$ in \cite[1966]{CFSU}, but it wasn't until Pengelley
constructed the spectrum $BoP$ (the subject of this paper) in
\cite[1982]{Peng}, that the stable homotopy type of $MSU$ (at 2) was given
as the product of copies of $BP$ and $BoP$.
The purpose of this paper is to compute the homology of the 
$\underline{BoP}_{\;k}$ and to conjecture the unstable homotopy 
type.
This completes the computation of all $H_*(\ul{MSU}_{\;*})$.

\section{Introduction}

In \cite{Peng}, Pengelley constructed a 2-local spectrum, $BoP$, such that the
special unitary cobordism spectrum localized at two, $MSU_{(2)}$, splits as
many copies of various suspensions of $BoP$ and $BP$, the Brown-Peterson spectrum.
To simplify notation, we use $\SBPP$ in place of
$\prod_{a \ge 0} \Sigma^{8a} {BP}$.
Andy Baker gives a stable cofibration
\begin{diag}
\[
\label{stable}
BoP \lra \SBPP \lra \Sigma^2 BoP
\]
\end{diag}

We use homology with $\Zq$-coefficients.
Our first theorem is about the homology of the spaces in the Omega spectrum for $BoP$:

\begin{thm}
\label{firstthm}
For $i \ge 2$, we have a short exact sequence of Hopf algebras:
\[
\Zq \ra H_*(\BoP{i}) \lra H_*(\BPP{i}) \lra H_*(\BoP{i+2}) \ra \Zq
\]
The homology of all three terms is polynomial on even degree generators when $i$ is
even and exterior on odd degree generators when $i$ is odd.  There is no torsion in
the $\Zp$-homology.
\end{thm}

\begin{remark}
Of course the homology of the middle term, $\BPP{i}$, is well known already, \cite{WSWThesis}.
Because all are either exterior or polynomial, the short exact sequence is split as algebras.
In addition, when $i > 2$ and even, the homology is bipolynomial (i.e. the cohomology is
also polynomial).
There is a temptation to believe that since the middle term is well-known, there must
be some sort of degree-by-degree induction algorithm that allows one to bootstrap the computation of
the homology of the other terms from that.  Regrettably, this is not the case.
The induction starts with the short exact sequence for $i=2$, but that is a hard won
exact sequence that depends on the degree by degree computations for all of the negative
spaces, with an added exotic transition from negative to positive spaces.
\end{remark}

By splicing these short exact sequences together
there is a novel corollary with no obvious use.  Using the composite maps
$\BPP{j} \ra  \BoP{j+2} \ra \BPP{j+2}$, we get

\begin{cor}
For $i \ge 2$ there is a long exact sequence of Hopf algebras
\[
\Zq \ra H_*(\BoP{i}) \lra H_*(\BPP{i}) \lra 
H_*(\BPP{i+2}) 
\lra 
\]
\[
H_*(\BPP{i+4}) 
\lra 
H_*(\BPP{i+6}) 
\lra \cdots
\]
\end{cor}

\begin{remark}
The object of interest to us is thus the zeroth homology of a chain complex
of well understood Hopf algebras.  
Unfortunately, the unstable maps here are not at all understood.
Stably, in (co)-homology, they are easy to see as everything is
given by sums of cyclic modules over the Steenrod algebra.
The actual stable map $r \colon \SBPP \ra \Sigma^2 \SBPP$ 
has the property that it covers $Sq^2$ on each copy of $BP$ and has
$r^2 = 0$.  This might well be enough to determine $r$, but it
is not to be messed with lightly.  Furthermore, it is unlikely
to give insight into the computation of the unstable maps with the
usual generators.
\end{remark}

Pengelley constructed a fibration $F \ra BoP \ra bo$ 
that gives a short exact sequence in homotopy.  It also has 
the property that
$bo$ carries all of the torsion homotopy of $BoP$ leaving $F$ with no torsion
in homotopy.

There is an old theorem of Conner and Floyd, \cite[Corollary 9.6, page 58]{CF}, 
that says (really for $MSU$, but essentially):
\begin{equ}
\[
\label{cf}
\BoP{0} \simeq \F{0} \times \bo{0}.
\]
\end{equ}
Notation is much simpler if we just assume that we are taking the 2-local
version of $bo$ and $\bo{i}$ throughout this paper.

We use this to get the following theorem covering the homology of the negative spaces:

\begin{thm}
\label{secthm}
For $i < 6$, we have 
\[
H_*(\BoP{i}) \simeq H_*(\F{i})\otimes H_*(\bo{i})
\]
where $H_*(\bo{i})$ is well known and $H_*(\F{i})$ is polynomial on even degree elements
when $i$ is even
and exterior on odd degree elements when $i$ is odd.  There is no torsion in the $\Zp$-homology
of $\F{i}$.
\end{thm}

We need some background in order to state our final theorem 
about homotopy type, but first some notation.  Let $\sA$ be the
mod 2 Steenrod algebra and $Q_k$ the Milnor primitives.
Define $\sA(k) = \sA/\sA(Q_0,Q_1,\ldots,Q_k)$ and
$\sA(2,k) = \sA/\sA(Q_0,Sq^2,Q_1,\ldots,Q_k)$.

\begin{thm}[\cite{WSWThesis2}, rephrased from cohomology to homology]
\label{wsw2}
\begin{enumerate}[leftmargin=2em] 
\ 
\item There exists a unique, up to homotopy,  irreducible $(k-1)$-connected H-space $Y_k$
which has $H_*(Y_k;\Zp)$ and $\pi_*(Y_k)$ free over $\Zp$.
\item If $Z$ is an H-space with $H_*(Z;\Zp)$ and $\pi_*(Z)$ free over
$\Zp$, then $Z \simeq \prod_i Y_{k_i}$.
\item There are spectra 
$BP\langle n \rangle$ 
with 
$H^*( BP\langle n \rangle) = \sA(n)$, 
$BP\langle n \rangle_* \simeq \Zp[v_1,v_2,\ldots,v_n]$ and $|v_i| = 2(2^i-1)$.
\item For $2^{j+1} -2 < k \le 2^{j+2} -2$, 
\ \ 
$Y_k
\simeq 
\BPn{j}{k}$.
\item
$\BPn{j}{2^{j+1}-2} 
\simeq 
\BPn{j-1}{2^{j+1}-2} 
\times
\BPn{j}{2^{j+2}-4} $.
\end{enumerate}
\end{thm}

\begin{thm}
\label{bosplit}
There is an irreducible splitting, not as H-spaces:
\[
\begin{array}{rcccll}
\BoP{6}&\simeq &
\bo{6}& \times &
\prod_{u \ge 0} &  Y_{8u+12} \\
& \simeq &
\bo{6}  &\times &
\prod_{2^{k-2} > u  \ge 0} &  Y_{2^{k+1}+8u+4} \\
& \simeq &
\bo{6}  &\times &
\prod_{2^{k-2} > u  \ge 0}& \BPn{k}{2^{k+1}+8u+4}
\end{array}
\]
\end{thm}

\begin{remark}
Looping down this splitting, we still have a splitting.  
After a few loops, it ceases to be irreducible.  
However, using theorem \ref{wsw2}[5.], the irreducible splittings
can be computed for as many loops as you want.  This is
a bad habit to get into though.
The awkward notation is a result of there being one 
$\SBPn{2}$,
two
$\SBPn{3}$,
four 
$\SBPn{4}$,
eight
$\SBPn{5}$, etc.
All as a result of theorem \ref{wsw2}[4.].
\end{remark}

Now to the speculative part of the paper.  

\begin{conj}
\label{conj}
There are spectra, $\SBoPn{n}$, with $\SBoPn{1} = bo$, with irreducible splittings
\[
\BoP{8n-2} \simeq 
\BoPn{n}{8n-2} 
\times \prod_{2^{k-2} -n + 1 > u \ge 0} 
\BPn{k}{2^{k+1} + 8(n+u) -4}
\]
\[
\BoPn{n+1}{8n-2} \simeq \BoPn{n}{8n-2} \times \prod_{2^{k-2}  \ge  n } 
\BPn{k}{2^{k+2}  -4}
\]
There are spectra, 
$\SBPPn{n}$, 
with $\SBPPn{1} = bu$, such that we have a stable
cofibration
\[
\SBoPn{n} \lra \SBPPn{n} \lra \Sigma^2 \SBoPn{n}
\]
inducing a short exact sequence on (co)homology.

Let $n > 2$.  Write $n = 2^K + a + 1$ with $0\le a < 2^K$.  Then
\[
H^*(\SBoPn{n}) \simeq
\bigoplus_{s=1}^{n-1}
\quad
\bigSigma_{K' \ge K}^{2^{K'+3+\epsilon} -8s} A(2,K'+2+\epsilon)
\]
and
\[
\SBPPn{n} \simeq
\bigWedge_{s=1}^{n-1}
\quad
\bigWedge_{K' \ge K}^{2^{K'+3+\epsilon} -8s} BP\langle K'+2+\epsilon \rangle
\]
where $\epsilon = 1$ if $0 < s \le a $ or $2^K + 1 = n-a \le s < n$ and $\epsilon = 0$ if
$a < s < n-a =2^K+1$.
\end{conj}

\begin{remark}
Particular thanks to Andy Baker for the cofibration \ref{stable}.  Thanks also to
Dave Johnson, David Pengelley, Vitaly Lorman, Vladimir Verchinine, and the referee.
Because of my interest in cobordism, I have always wanted to study Pengelley's $BoP$ but never found
the time.  I was inspired by a vague recollection of a 2001 email from 
Mike Slack suggesting that there is no torsion in 
spaces in the Omega spectrum.  After proving these results, I went back and reviewed my emails from 
Mike
and found much more than that.  
I discovered a
1995 email that used the same notation 
$\SBoPn{n}$.
The email does not specify the detail in the conjecture \ref{conj}, but
went further in one sense.  In \cite{Slack}, Slack proved that if
an infinite loop space had no odd torsion in its integral homology, 
then it also had no odd torsion in homotopy.  This did not hold
for $p=2$ as is illustrated by the spaces like $\ul{bo}_{\; 2}$,
and now $\ul{BoP}_{\;k}$, $k \ge 2$.
However, Michael Slack conjectured that you could classify all
infinite loop spaces with two-primary torsion-free homology in
terms of products of the spaces
$\BoPn{n}{k} $
that are irreducible and torsion-free in homology.
He left mathematics before he completed his proofs and wrote things up.  If he hadn't left
mathematics, I'm sure everything in this paper, and more, would have been done by Michael Slack 15 years
ago.
\end{remark}

In section 
\ref{ha-bss}, 
we do a quick review of what we need about Hopf algebras and 
the bar spectral sequence.
We review what we need about $BoP$ 
in section 
\ref{rev-bop}.
In section 
\ref{fi} 
we do our computation for the spaces $\F{i}$, $i < 6$, and
prove theorem \ref{secthm}. In
section 
\ref{hbop}, 
we move on to $\BoP{i}$, $i \ge 2$ and prove theorem \ref{firstthm}.
We prove the splitting, theorem \ref{bosplit}, in section \ref{secbosplit}
and discuss the conjecture \ref{conj} in the last two sections.

\section{Hopf algebras and the bar spectral sequence}
\label{ha-bss} 

All of the spaces we deal with are spaces in Omega spectra, i.e. $ Z = \{\ZZ{i}\}$ with
$\Omega \ZZ{i+1} \simeq \ZZ{i}$.  The spectra are all connective, i.e. $\pi_*^S(Z) = 0$
for $* < 0$.

As a result, the mod 2 homology of all our spaces, $H_*(\ZZ{i})$, are all bicommutative,
biassociative graded (and sometimes bigraded) Hopf algebras.  When we use the zero
component, $\ZZZ{i}$, the Hopf algebra is connected.  The standard reference for
Hopf algebras is, of course, \cite{MM}.

Because we know the homotopy of our spaces, we know the zeroth homology of the spaces
in the Omega spectra.  We have $H_0(\ZZ{*}) \simeq \Zq[Z^*]$, the group ring on
the coefficients.  More precisely, $H_0(\ZZ{i}) \simeq \Zq[Z^i] \simeq \Zq[\pi_{-i}^S(Z)]$.
Then we have $H_*(\ZZ{i}) \simeq \Zq[Z^i]\otimes H_*(\ZZZ{i})$.
Except for $bo$ and $BoP$, all of our spaces have homotopy only in even degrees.
As a result, the above isomorphism degenerates when $i$ is odd because $\ZZ{i} = \ZZZ{i}$.

Our main tool is the bar spectral sequence.  We state what we need here.

\begin{thm}[\cite{RS}]
There is a first quadrant homology spectral sequence of Hopf algebras going from
$H_*(\ZZ{i-1})$ to $H_*(\ZZ{i})$ with 
\[
E^2_{*,*} \simeq \Tor^{H_*(\ZZ{i-1})}_{*,*}(\Zq,\Zq) \Rightarrow H_*(\ZZ{i})
\]
\[
d_r \colon E^r_{u,v} \lra E^r_{u-r,v+r-1}
\]
\end{thm}

\begin{remark}
Because we are dealing with infinite loop spaces we do not run into weird convergence
problems associated with problematic fundamental groups.  In fact, we will do away
with the use of the zero component, $\ZZZ{i}$, and, instead, use the convention that we
think of the group ring $\Zq[\Zp]$ as a polynomial algebra on one generator.
When there is no torsion in the homotopy,
the use of this convention gives
$\Tor^{H_0(\ZZ{i})}_{*,*}(\Zq,\Zq)$ is an exterior algebra on the suspensions
of the $\Zq[\Zp]$ generators in $\Zq[Z^i]$, all located in $(1,0)$.  This follows from
the Hurewicz isomorphism in degree 1 if not the algebra.
\end{remark}

We actually only need the spectral sequence in very limited situations.
Keep in mind that $\Tor$ commutes with tensor products.

\begin{prop}
\label{free}
If $H_*(\ZZ{i-1})$ is polynomial on even degree generators for 
degrees $* < j$, then $H_*(\ZZ{i})$ is an exterior algebra 
on the suspensions of the generators
in degrees $* \le j$.
If $H_*(\ZZ{i-1})$ is exterior on odd degree generators for 
degrees $* < j$, then $E^2 = E^\infty H_*(\ZZ{i})$ is 
an even degree divided power algebra on the suspensions of the generators
in degrees $* \le j$.
\end{prop}

\begin{proof}
To compute $\Tor$ you need a resolution.  Since we are only using degrees $* < j$,
it follows that our input gives all $E^2_{u,v}$ with $v < j$.
Also, ignoring the $\Zq$ in degree zero, there is nothing in the zero filtration.
Since $u > 0$, this gives all $E^2_{u,v}$ with total degree $u+v \le j$. 

If $i-1$ is even, we take
$\Tor$ of a polynomial algebra with even degree 
generators $x_{2i}$ to get an exterior algebra
on the elements that are  suspensions of the generators in bidegree $(1,2i)$.
Since they are in filtration 1, there are no differentials on them.  
Because of this, there can be no differentials from degree $(u,v)$ when $v<j$.
Any differential that hits something in total degree $* \le j$, must come from
$(u,v)$ with $u > 2$ and $u+v \le j+1$, but this requires $v < j$, so
there can be no differentials interfering.  We have $E^\infty$ 
for $ * \le j$ is an exterior algebra on odd degree generators.  
There can be no extensions because all that can happen is an odd exterior 
generator squares to an even generator, but there are none in degrees $* \le j$.

If $i-1$ is odd,
$\Tor$ of an exterior algebra with odd degree 
generators $x_{2i+1}$ is a divided power algebra, $\Gamma[\sigma x_{2i+1}]$,
on the elements $\sigma x_{2i+1}$ in $(1,2i+1)$.
Thus $\Tor$, up through degrees $* \le j$ is even degree so there can be no differentials
on them.  Just as above, they cannot be hit by differentials.
There can still be extension problems with respect to squaring elements.
\end{proof}

\begin{remark}
In the case of the spectral sequence above for $\BPP{i}$, we know that the homology for
$i$ odd is exterior and for $i$ even polynomial.  That means that all of the possible
extension problems in the spectral sequence for $i$ even must be solved. 
\end{remark}

\section{Review of $BoP$}
\label{rev-bop}

We need to collect a few facts.
Recall our notation,
 $\SBPP  =
\prod_{a \ge 0} \Sigma^{8a} {BP}$.

\begin{prop}[from A. Baker]
\label{baker}
There is a stable cofibration
\[
BoP \lra \SBPP \lra \Sigma^2 BoP
\]
This
is torsion free and short exact on (co)homology.
\end{prop}

Note that there is no torsion in homology for these spectra because they are even degree.

\begin{proof}
We smash $BoP$ with the cofibration
\[
S^1 \lra S^0 \lra C(\eta)
\]
to get 
\[
\Sigma BoP \lra BoP \lra BoP \wedge C(\eta)
\]
or
\[
BoP \lra 
BoP \wedge C(\eta)
\lra
\Sigma^2 BoP 
\]
All we need to do is show that 
\[
BoP \wedge C(\eta) 
\simeq
\SBPP
\]
Let $A$ be the mod 2 Steenrod algebra and $Q_i$ the Milnor primitives, \cite{MilA}.
From \cite{Peng} we know the mod 2 cohomology
\[
H^*(BoP) \simeq \oplus_{a \ge 0} \Sigma^{8a} A/A(Q_0,Sq^2, Q_1,Q_2,Q_3, \ldots)
\]
The usual way to write this is to take out $Q_0$ on both the left and right, but that
is equivalent to taking out all of the $Q_i$ on the right.  Then $Sq^2$ is taken out
on the right as well.  For future use we prefer this notation.
We know $H^*(C(\eta))$ just has cells in degree 0 and 2 connected by $Sq^2$.  
This gives
\[
H^*(BoP\wedge C(\eta)) \simeq H^*(BoP)\otimes H^*(C(\eta))
\simeq
 \oplus_{a \ge 0} \Sigma^{8a} A/A(Q_0, Q_1,Q_2,Q_3, \ldots)
\]
This last is the cohomology of $\SBPP$, and anything with such cohomology is 
a product of $BP$ spectra, \cite{BP}.
The short exact sequence follows.
\end{proof}

We also make use of the old result (of R. Wood, see \cite[2.3.1]{KLW} for a discussion
of references) that comes about in a similar way. 

\begin{prop}
There is a stable cofibration
\[
bo \lra bu \lra \Sigma^2 bo
\]
This
is short exact on $\Zq$-(co)homology.
\end{prop}

We always assume our $bo$ and $bu$ to be the 2-local versions and will suppress
using the notation $bo_{(2)}$ and $bu_{(2)}$.

We need to define two spectra, $F$ and $X$. 
$X$ is easy.

\begin{lemma}[for $BoP$, \cite{Peng}]
There are stable cofibrations that define spectra $F$ and $X$ with torsion-free
even degree homotopy.
They give short exact sequences in their homotopy groups.
\[
F \lra BoP \lra bo
\]
\[
X \lra \SBPP \lra bu
\]
\end{lemma}

\begin{proof}
The spectrum $bu$ at $p=2$ is sometimes called 
$BP\langle 1 \rangle$ 
and there is a map
$BP \ra 
BP\langle 1 \rangle$, see \cite{JW2}.
The other copies of $BP$ do not enter in.
\end{proof}

This all gives rise to horizontal and vertical cofibrations:
\begin{diag}
\[
\label{bigone}
\xymatrix{
F \ar[r] \ar[d] & BoP \ar[r] \ar[d] &  bo \ar[d] \\
X \ar[r] \ar[d] & \SBPP \ar[r] \ar[d] & bu \ar[d] \\
\Sigma^2 F \ar[r]  & \Sigma^2 BoP \ar[r]  & \Sigma^2 bo  \\
}
\]
\end{diag}

We need some unstable information to make our computations work.

\begin{thm}[ $i=0$, Corollary 9.6, page 58, and following comments \cite{CF}]
\label{cf66}
\[
\BoP{i} \simeq 
\F{i} 
\times
\bo{i}
\qquad i \le 6
\]
\[
\BPP{i} \simeq \X{i} \times \bu{i} \qquad i \le 6
\]
\end{thm}

\begin{proof}
Of course Conner and Floyd did not use $BoP$ and $BP$ in their work because they did
not exist at the time.  The above proposition is the modern interpretation.
They only proved these cases
for $i=0$ and that is all we will use.  
Of course, the $i < 0$ cases follow.
Only in the end will we improve on this and
give the $\BoP{i}$ splittings for $0 < i \le 6$.
The $BP$-$bu$ case is already known from \cite{WSWThesis2} which can be used
to give  the entire
homotopy type of
$\X{i}$ for $i \le 8$. 
\end{proof}

We need just a few more things.

\begin{lemma}
\label{splithomotopy}
The homotopy groups give
a split short exact sequence of free $\Zp$-modules for the fibration
\[
\F{i} \lra \X{i} \lra \F{i+2}
\]
\end{lemma}
\begin{proof}
In the diagram \ref{bigone}, all of the $\Zq$-groups are in $BoP$ and $bo$.
All of the $\Zp$-free groups in the other spectra are in even degrees.  The
result follows.
\end{proof}

\begin{remark}
It is perhaps unfair to assume too much pre-existing knowledge.  To correct
that oversight, we give the homotopy groups of the well-known objects in diagram \ref{bigone}.
\[
BP_* \simeq \Zp[v_1, v_2,\ldots] \qquad |v_n| = 2(2^n-1)
\qquad \qquad
bu_* \simeq \Zp[v_1]
\]
\[
bo_{4i} \simeq \Zp
\qquad
bo_{8i+1} \simeq \Zq
\simeq
bo_{8i+2} 
\]
\end{remark}

In addition, we need some information about homology.

\begin{prop}
\label{xprop}
The homology, $H_*(\BPP{i})$, is torsion free and is polynomial on even degree
generators for $i$ even and exterior on odd degree generators for $i$ odd.
The homology, $H_*(\X{i})$, $i < 6$, is torsion free and is polynomial on even degree
generators for $i$ even and exterior on odd degree generators for $i$ odd.

There is a short exact sequence of polynomial Hopf algebras
\[
\Zq \lra H_*(\bo{2}) \lra H_*(\bu{2}) \lra H_*(\bo{4}) \lra \Zq
\]
\end{prop}

\begin{proof}
The homology of $\BPP{i}$ is known from \cite{WSWThesis} and that for $\X{i}$
follows from the splitting, theorem \ref{cf66}.  The given short exact sequence is well-known,
but see remark \ref{homologies} below as well.
\end{proof}

\begin{remark}
This can be refined to see that $\X{i}$ has no torsion for $i \le 8$, but there the homology is
not polynomial.
\end{remark}

\begin{remark}
\label{homologies}
Again, it seems only fair to recall some of the well known results.  $\bu{2} = BU$, and, as
such, the homology is just $H_*(BU)\simeq \Zq[x_{2i}]$.  The spectrum $bo$ is more
complicated with $\bo{0} = Z \times BO = \KO{0}$.  By Bott periodicity 
we always have $\bo{i} = \bo{i-8} =
\KO{i}$ for $i < 4$ (in particular, for all negative $i$).
Let $P$ denote a polynomial algebra and $E$ and exterior algebra.  Also let $z_i$ denote
an $i$-th degree element.
We know that 
\[
H_*(\bo{0})\simeq \Zq[Z]\otimes P[x_i] \quad i > 0
\qquad
H_*(\bo{1})\simeq P[x_{2i+1}]
\qquad
H_*(\bo{2})\simeq P[x_{4i+2}]
\]
\[
H_*(\bo{3})\simeq E[x_{4i+3}]
\qquad
H_*(\KO{4})\simeq \Zq[Z]\otimes H_*(\bo{4}) \simeq \Zq[Z] \otimes P[x_{4i}]
\quad i > 0
\]
\[
H_*(\KO{5})\simeq E(x_1)\otimes H_*(\bo{5}) \simeq E(x_1) \otimes E[x_{4i+1}] \quad i > 0
\]
\[
H_*(\KO{6})\simeq E(x_{2^k})\otimes H_*(\bo{6}) \simeq E(x_{2^k}) \otimes E[x_{2i}] \quad 2i \ne 2^k 
\]
\[
H_*(\KO{7})\simeq E(x_i)
\]
This is all conveniently written down in \cite{CS} and \cite[Theorem 25.2]{NituER2}.
\end{remark}

\section{$H_*(\F{i})$, $i < 6$}
\label{fi} 

The goal of this section is to prove the following theorem.

\begin{thm}
\label{xf}
For $i \le 6$, 
the fibration of Lemma
\ref{splithomotopy}
gives a short exact sequence of Hopf algebras
\[
\Zq \ra H_*(\F{i}) \lra H_*(\X{i}) \lra H_*(\F{i+2}) \ra \Zq 
\]
When $i$ is odd, all three are exterior algebras on odd degree generators.
When $i < 6$ is even, all three are polynomial algebras on even degree generators.
\end{thm}

\begin{remark}
This, plus the Conner-Floyd result, equation \ref{cf} and theorem \ref{cf66}, 
gives theorem \ref{secthm} for $i \le 0$.
We will have to return to theorem \ref{secthm} later for the $0 < i < 6$.
\end{remark}

\begin{proof}
We already have the homology of $\X{i}$ by proposition \ref{xprop}.
We do our proof by induction on degree.
We know that on the zero-degree homology we have such
an exact sequence because there the groups are just given by the homotopy groups
where we have exactness from 
lemma \ref{splithomotopy}.
Exactness for
$H_1(-)$
follows from our convention on $\Tor$ or just using the Hurewicz isomorphisms.
$H_2(-)$ and exactness follows in the same way.
One can even go one step further and get $H_3(-)$.
This starts our induction.

For our induction, we assume the result for degrees $* < j$ and we will show
it holds for degree $j$ as well.
By induction, we have the result for
\[
\Zq \ra H_*(\F{i-1}) \lra H_*(\X{i-1}) \lra H_*(\F{i+1}) \ra \Zq \qquad * < j
\]
Since these are either exterior or polynomial, this exact sequence is split
as algebras.  Computing $\Tor$ only depends on the algebra structure, so
in our range $* \le j$, we know from proposition \ref{free} that we get a short exact sequence
of Hopf algebras 
on the $E^2 = E^\infty$ terms.

When $i$ is odd, we are done because the results are exterior.  When $i$ is even,
we get the inclusion $H_*(\F{i}) \ra H_*(\X{i})$.  Since the second one is
polynomial, the first must also be polynomial (Hopf algebra structure requires
this).

However, we have not yet shown that $H_*(\F{i+2})$ is polynomial for $i$ even, $i < 6$. 
To see this,
we use the following short exact sequence we have by induction 
\[
\Zq \ra H_*(\F{i+1}) \lra H_*(\X{i+1}) \lra H_*(\F{i+3}) \ra \Zq \qquad * < j
\]
Since $i$ is even and 
$i < 6$,
$i+1 < 6$ fits our induction hypothesis.
Once again we apply $\Tor$ to this split short exact sequence of algebras to
get a short exact sequence on the
$E^2 = E^\infty$ terms.  Now we have 
the inclusion $H_*(\F{i+2}) \ra H_*(\X{i+2})$ forcing $H_*(\F{i+2})$ to
be polynomial.  
If $i = 4$, this gives 
\[
\Zq \ra H_*(\F{6}) \lra H_*(\X{6}) \lra H_*(\F{8}) \ra \Zq \qquad * \le  j
\]
In this case, the right hand term need not be polynomial, but is still even
degree (and cofree).
\end{proof}

\begin{proof}[Proof of theorem \ref{secthm}]
We
use the 
Atiyah-Hirzebruch spectral sequence 
(AHSS)
for the fibration $\F{i} \ra \BoP{i} \ra \bo{i}$, $0 < i \le 6$.
The $E^2$ term is 
\[
E^2 \simeq H_*(\F{i})\otimes H_*(\bo{i}) \Rightarrow H_*(\BoP{i})
\]
These spectral sequences all collapse.  
For $i=2,4,$ and $6$ they collapse because they are even degree.
For $i=3$ and $5$, everything is exterior on odd degree generators so there is
no place for a differential to go.
The $i=1$ case is a bit more complicated.  
The $H_*(\F{1})$ term is exterior on odd degree generators and $H_*(\bo{1})$
is polynomial on odd degree generators.  Differentials must start on the odd
degree polynomial generators of $H_*(\bo{1})$ and end up in an even degree, so,
as with the exterior case, there is nothing to hit.
In addition, there are no algebra extension problems for $0 < i < 6$.
For $1 < i < 6$ this is because the $E^2$ term is either exterior on odd degree
generators or polynomial on even degree generators.
In the case of $i=1$, we have to look at the maps 
$H_*(\F{1})\ra H_*(\BoP{1})\ra H_*(\bo{1})$ to see that the exterior elements from
$H_*(\F{1})$ cannot have extension problems and since $H_*(\bo{1})$ is polynomial,
this splits as algebras.
\end{proof}

\begin{remark}
\label{interest}
Where this gets interesting  is for the $i=6$ case.
We have the short exact sequence
$H_*(\F{6})\ra H_*(\BoP{6})\ra H_*(\bo{6})$. We will soon see that $H_*(\BoP{6})$
is polynomial.  However, $H_*(\bo{6})$ is exterior.  Later we will even show that
$\BoP{6} \simeq \F{6} \times \bo{6}$, just not as H-spaces.  So, the squares of
the exterior generators in $H_*(\bo{6})$ must be non-zero in $H_*(\F{6})$.
\end{remark}

\section{$H_*(\BoP{i})$, $i \ge 2$}
\label{hbop}

To begin our induction for the proof of theorem \ref{firstthm}, we need the following.

\begin{thm}
There is a short exact sequence of polynomial Hopf algebras:
\[
\Zq \ra H_*(\BoP{2}) \lra H_*(\BPP{2}) \lra H_*(\BoP{4})  \ra \Zq
\]
\end{thm}

\begin{proof}
Consider the diagram 
\begin{diag}
\[
\label{onlytwo}
\xymatrix{
\F{2} \ar[r] \ar[d] & \BoP{2} \ar[r] \ar[d] &  \bo{2} \ar[d] \\
\X{2} \ar[r] \ar[d] & \BPP{2} \ar[r] \ar[d] & \bu{2} \ar[d] \\
 \F{4} \ar[r]  &  \BoP{4} \ar[r]  &  \bo{4}  \\
}
\]
\end{diag}
Take the $E^2$-terms of the  AHSS for the horizontal fibrations and maps to get
\[
H_*(\F{2})\otimes H_*(\bo{2})
\lra
H_*(\X{2})\otimes H_*(\bu{2})
\lra
H_*(\F{4})\otimes H_*(\bo{4})
\]
These spectral sequences are all even degree so collapse, 
are polynomial so there are no extension problems, and the maps form the required 
short exact sequence
by theorem  \ref{xf} and proposition \ref{xprop}.
\end{proof}

\begin{proof}[Proof of theorem \ref{firstthm}]
We proceed pretty much as we did with the short exact sequences for negative spaces.
We inductively assume our result for all $* < j$.
We can certainly start our induction because we know the result for $j = 2$
(the only groups here are 
$H_2(\BoP{2}) \simeq H_2(\BPP{2}) \simeq \Zq$.  All others are zero.)
So, we do our induction on $j$.  If $i$ is odd, $i > 2$, we know
\begin{diag}
\label{iminus}
\[
\Zq \ra H_*(\BoP{i-1}) \lra H_*(\BPP{i-1}) \lra H_*(\BoP{i+1}) \ra \Zq
\]
\end{diag}
\noindent
satisfies our inductive hypothesis for $* < j$.  

Compute $\Tor$ on this split polynomial short exact
sequence to get a split short exact (collapsing) bar spectral 
sequence of exterior algebras, by proposition \ref{free}.  
This completes the $i$ odd cases.

If $i$ is even, we can assume $i > 2$ because we have done all of $i=2$ already.

We have diagram \ref{iminus}, for $* < j$,
where the $i-1$ is now odd (and $>2$).  Taking $\Tor$ for the bar spectral sequence we get a short exact
sequence of collapsing even degree divided power algebras (proposition \ref{free}).  The middle term is
polynomial so the left hand side, $H_*(\BoP{i})$ is as well, for degrees $* \le j$.  
The problem is to show that the right hand term,
$H_*(\BoP{i+2})$, is polynomial.  However, we have 
\[
\Zq \ra H_*(\BoP{i+1}) \lra H_*(\BPP{i+1}) \lra H_*(\BoP{i+3}) \ra \Zq
\]
satisfies our induction hypothesis for $* < j$.  
(This is why we can't just do an induction on $i$ for all degrees at once.)
These are all odd, so exterior.  Taking
$\Tor$ of this gives our usual short exact sequence of Hopf algebras, divided power
algebras on even degrees (so collapse) through degrees $* \le j$.  Because the
middle term is polynomial, the left hand term is too, i.e.
$H_*(\BoP{i+2})$.

This completes the proof.
\end{proof}

\section{Proof of the splitting}
\label{secbosplit}

The proof of theorem \ref{bosplit} is mostly about knowing how to write the homotopy
groups.  To this end we have:

\begin{lemma}
As graded abelian groups
\[
BoP_* \simeq bo_* \oplus_{k \ge 2} \oplus_{u=0}^{2^{k-2} -1 } \Sigma^{2^{k+1}+8u-2}
\SBPn{k}_*
\]
\end{lemma}

\begin{proof}
$BoP_*$ and $bo_*$ have the same $\Zq$ groups and everything else is torsion free.
As a result, all we have to do is show both sides are the same rationally.
The  Poincare series for the rational homotopy of the $\SBPP_*$ 
is: 
\[
\frac{1}{1-x^8} \times \prod_{j > 0} \frac{1}{1-x^{2(2^j-1)}}
\]
From this and the rational short exact sequence of display \ref{stable} and lemma \ref{splithomotopy}, 
we can read off the Poincare 
series for the homotopy of $BoP$:
\[
\frac{1}{1-x^8} \times
\frac{1}{1-x^4} 
\times \prod_{j > 1} \frac{1}{1-x^{2(2^j-1)}}
\]
The rational Poincare series for the right hand side (RHS) is 
\[
\frac{1}{1-x^4} 
+
\sum_{k \ge 2} x^{2^{k+1}-2} \sum_{u=0}^{2^{k-2} -1 } 
\frac{x^{8u}}{\prod_{j > 0}^k (1-x^{2(2^j-1)})} 
\]
We need to show these two are the same.
Multiply both sides by the denominator on the left hand side, i.e. 
\[
(1-x^8) 
(1-x^4) 
 \prod_{j > 1} (1-x^{2(2^j-1)})
\]
This would leave a $1-x^2$ in the denominator of the RHS, but we can divide
that into $1-x^4$ to get $1+x^2$ in the numerator.  We want to show:
\[
1=
(1-x^8) 
 \prod_{j > 1} (1-x^{2(2^j-1)})
+
(1-x^8)(1+x^2)
\sum_{k \ge 2} x^{2^{k+1}-2} 
{\prod_{j > k} (1-x^{2(2^j-1)})} 
(\sum_{u=0}^{2^{k-2} -1 } 
x^{8u})
\]
But
\[
\sum_{u=0}^{2^{k-2} -1 } 
x^{8u}
=
\frac{1 - x^{2^{k+1}}}{1-x^8}
\]
The $(1-x^8)$ cancels out and we have the RHS is
\[
(1-x^8) 
 \prod_{j > 1} (1-x^{2(2^j-1)})
+
\sum_{k \ge 2} x^{2^{k+1}-2} 
(1+x^2)
(1 - x^{2^{k+1}})
{\prod_{j > k} (1-x^{2(2^j-1)})} 
\]
We need three definitions to complete the proof.  Let
\[
A_s = \sum_{k \ge s} x^{2^{k+1} -2 } (1+x^2)(1-x^{2^{k+1}}) \prod_{j > k} (1-x^{2(2^j-1)})
\]
\[
B_s = (1-x^{2^{s+1}}) \prod_{j \ge s}(1 - x^{{2(2^j -1)}})
\]
\[
C_s = x^{2(2^s-1)} (1+x^2) (1 - x^{2^{s+1}}) \prod_{j > s} (1 - x^{2(2^j-1)}) 
\]
We have two straightforward identities:
\[
RHS = B_2 + A_2
\qquad
\qquad
A_s = C_s + A_{s+1}
\]
We plan to show that
\[
B_{s+1} = B_s + C_s
\]
This would give us
\[
B_s + A_s = B_s + C_s + A_{s+1} = B_{s+1} + A_{s+1}
\]
We note that $A_s$ has higher and higher powers of $x$ in it
and that $B_s -1$ also has higher and higher powers of $x$ in it.
Take the limit as $s$ goes to infinity and we see that the RHS $=1$.
We still have to show the inductive step.
So, we compute 
$B_s + C_s$.
We have
\[
 (1-x^{2^{s+1}}) \prod_{j \ge s}(1 - x^{{2(2^j -1)}})
+ x^{2(2^s-1)} (1+x^2) (1 - x^{2^{s+1}}) \prod_{j > s} (1 - x^{2(2^j-1)}) 
\]
Factor out 
\[
\prod_{j > s} (1 - x^{2(2^j-1)}) 
\]
To get
\[
\Big[ 
 (1-x^{2^{s+1}}) (1 - x^{{2(2^s -1)}})
+ x^{2(2^s-1)} (1+x^2) (1 - x^{2^{s+1}}) 
\Big]
\prod_{j > s} (1 - x^{2(2^j-1)}) 
\]
Looking at what is in the brackets, we have
\[
 (1-x^{2^{s+1}}) \Big[ (1 - x^{{2(2^s -1)}})
+ x^{2(2^s-1)} (1+x^2)\Big] 
= 
\]
\[
 (1-x^{2^{s+1}})  (1 + x^{2^{s+1}}  )
= 
1 - x^{2^{s+2}}
\]
which is exactly what we needed to finish the proof.
\end{proof}

We need to show that 
\begin{claim}
The
$\BPn{k}{2^{k+1}+8u+4}$ of theorem \ref{bosplit} are irreducible.
\end{claim}
\begin{proof}
By theorem \ref{wsw2}[(4)], all we need to do is show that
$2^{k+1} - 2 < 2^{k+1}+8u+4 \le 2^{k+2} -2$ 
when $2^{k-2} > u \ge 0$.  This is a simple exercise.
\end{proof}

\begin{proof}[Proof of theorem \ref{bosplit}]
From remark \ref{interest}, we have the short exact sequence
\[
\Zq \ra H_*(\F{6}) \lra H_*(\BoP{6}) \lra H_*(\bo{6}) \ra \Zq
\]
We have shown that $\F{6}$ has no torsion in homology or homotopy,
so 
theorem \ref{wsw2}[(2)] gives us the homotopy type of $\F{6}$ as the
$\BPn{k}{i}$ part of theorem \ref{bosplit}.
We need a map
\[
\BoP{6} \lra
\bo{6}  \times 
\prod_{2^{k-2} > u  \ge 0} \BPn{k}{2^{k+1}+8u+4}
\]
that gives an isomorphism on homotopy.  We have the map to $\bo{6}$.
Using the proofs of theorem \ref{wsw2} in \cite{WSWThesis2} we can construct our
other maps.  If we have a double no-torsion H-space, $Z$, to get a map
to $Y_k$, it is enough to start with a map $Z \ra K(\Zp,k)$.  All the
k-invariants are torsion (trivially true for any H-space) 
and so zero, so we can lift the map to $Z \ra Y_k$.
Picking our $k_i$ appropriately, this gives the required map $Z \ra \prod Y_{k_i}$.
In our case, we have the maps from $\F{6}$, but by the short exact sequence,
we can lift cohomology classes of $\F{6}$ to $\BoP{6}$.  $\BoP{6}$ has no
torsion in cohomology, so we can get our lifts to $Y_k$ even though $\BoP{6}$
has 2-torsion in homotopy.  This now gives us our maps and our homotopy equivalence.
\end{proof}

\section{Discussion of the conjecture, Part 1}
\label{secconj}

The easiest way to illustrate our evidence for our conjecture \ref{conj} is
to look at the $n=1$ case
\[
\BoPn{2}{6} \simeq \bo{6} \times \prod_{2^{k-2}  \ge  1 } 
\BPn{k}{2^{k+2}  -4}
\]
Since all the terms on the right hand side are even degree and have no torsion in homology, we would
want the property to hold if we delooped  twice to get our $\BoPn{2}{8}$.
We know that the splitting is not as H-spaces though because the homology of $\bo{6}$
is an exterior algebra.  If our $\BoPn{2}{8}$ exists, we need the squares of the
elements in $H_*(\bo{6})$ to lie in the other terms.  We know that they must be
somewhere in our splitting for $\BoP{6}$, we just want them in these specific 
spaces.

The  homology of
$\BPn{k}{2^{k+2}  -4}$ is known to be polynomial from \cite{WSWThesis}, but we are
going to need to get very technical and use the results of 
\cite{RW:HR} where we write down generators for this homology.  Since this is a speculative
part of the paper, we will not go into the necessary lengthy review of \cite{RW:HR}
needed, but the interested reader can pursue this on their own.

We have, from  remark
\ref{homologies},
$H_*(\bo{6}) \simeq E[x_{2j}]$,  $2j \ne 2^k$.
If we write 
\[
j = 2^{s_1} + 2^{s_2} + \cdots + 2^{s_k} \qquad \quad 0 \le s_1 < s_2 < \cdots < s_k  \qquad k > 1
\]

We conjecture that 

\[
x_{2j}^{*2} \simeq
b_{(s_1)}^2
b_{(s_2 -1)}^4
b_{(s_3 -2)}^8
\cdots
b_{(s_k -k + 1)}^{2^k}
\in
H_*(\BPn{k}{2^{k+2}  -4})
\]

If we have this relation and double suspend to the homology of
$\BPn{k}{2^{k+2}  -2}$, because suspension kills star products,
the elements 
\[
b_{(s_0)}
b_{(s_1)}^2
b_{(s_2 -1)}^4
b_{(s_3 -2)}^8
\cdots
b_{(s_k -k + 1)}^{2^k}
\in
\BPn{k}{2^{k+2}  -2}
\quad
\text{with}
\quad 
s_0 \le s_1
\]
would be zero.
The reason this looks good is because  the homology of 
$\BPn{k}{2^{k+2}  -2}$ is not polynomial.  Better, using \cite{RW:HR}, we can
see that the above elements are exactly the elements whose squares are zero, and
we cannot have such elements in the homology of
$\BoP{8}$ because it is polynomial.  

Every 8 deloopings we find ourselves in the same position and we have to incorporate
more and more of the $\BPn{k}{i}$ into our $\BoPn{n}{i}$ to maintain polynomial algebras with
no torsion.

\section{Discussion of the conjecture, Part 2}
\label{secconj2}

The conjecture for
$\SBPPn{n}$ follows from the conjecture for
$H^*(\SBoPn{n})$. 
This later conjecture comes from a hypothetical inductive (on $n$)
computation
based on the conjectured unstable splitting of 
$\SBoPn{n}$ in terms of 
$\SBoPn{n-1}$ and the $BP\langle k \rangle$.
It is what delayed the paper so long.

There are a few observations worth noting.  First, as $n$ goes
to infinity, 
$H^*(\SBoPn{n})$ goes to $H^*(BoP)$.

Next, if we look at the summand on degree $8q$, the first $n$
it appears with is $n = 2^J + 1 -u$ if $q = 2^J + u$, $0\le u < 2^J$.

In conclusion, and enterprising individual might compute the homology
of all our spaces giving names to the generators and then find the
splitting using this information.  It would seem to be a lot of work.


\end{document}